\documentclass{article}
\usepackage{amsfonts}
\usepackage{amsmath}

\newtheorem{theorem}{Theorem}

\newtheorem{corollary}[theorem]{Corollary}

\newtheorem{proposition}[theorem]{Proposition}
\newtheorem{remark}[theorem]{Remark}

\newenvironment{proof}[1][Proof]{\noindent \textbf{#1.} }{\  \rule{0.5em}{0.5em}}
\input{tcilatex}
\begin{document}

\begin{center}
{\LARGE Binomial sums and properties of the Bernoulli transform}

\  \  \ 

\textbf{Laid ELKHIRI}$^{1}$\textbf{, Miloud MIHOUBI}$^{2}$\textbf{\ and
Meriem MOULAY}$^{3}$, \textbf{\medskip }

$^{1}$Tiaret University, Faculty of Material Sciences, RECITS Laboratory,
USTHB, Algiers, Algeria.\\[0pt]
$^{2,3}$USTHB, Faculty of Mathematics, RECITS Laboratory, PO 32, El-Alia,
16111, Algiers, Algeria. \medskip \\[0pt]

$^{1}$\texttt{laid@univ-tiaret.dz, laidmaths82@gmail.com}

$^{2}$\texttt{mmihoubi@usthb.dz, \ miloudmihoubi@gmail.com}

$^{3}$\texttt{m.moulay@usthb.dz, \ meriemro3@gmail.com}
\end{center}

\  \  \  \newline
\textbf{Abstract. }In this paper, we study the binomial sum $S_{n}(q):=%
\overset{n}{\underset{k=0}{\sum }}a_{k}\binom{n}{k}\left( 1-q\right)
^{k}q^{n-k}$ for a given sequence $\left( a_{n}\right) $ of real or complex
numbers. We express $S_{n}(q)$ in function of the powers of $q,$ and, we
explicit it when the sequence $\left( a_{n}\right) $ is the sequence of
Fibonacci numbers, Laguerre polynomials, Meixner polynomials, binomial
coefficients and the sequence $\left[ n\right] _{p}.$ We establish later
some properties, relations, probabilistic interpretations and generating
functions between $S_{n}(q)$ and $S_{n}(x+q-xq).$ Further identities related
to Appell polynomials are also given in the last of the paper.


\noindent \textbf{Math. Subj. Class.: }11B50; 11C08; 11B39; 11B83.

\section{Introduction}

The Bernoulli transform of a sequence of real numbers $a=\left( a_{n}\right) 
$ is the sequence $\left( S_{n}(q)\right) $ defined by%
\begin{equation*}
S_{n}(q):=S_{n}^{\left( a\right) }(q)=\overset{n}{\underset{k=0}{\sum }}a_{k}%
\binom{n}{k}\left( 1-q\right) ^{k}q^{n-k},
\end{equation*}%
and it used by Flajolet \cite{Fla} in 1999 in aim to estimate asymptotically 
$S_{n}^{\left( a\right) }(q).$ He established interesting approximations for
several choices of the sequence $\left( a_{n}\right) ,$ such those for which 
$S_{n}^{\left( a\right) }(q)$ is asymptotically equivalent to $%
a_{\left \lfloor n\left( 1-q\right) \right \rfloor },$ where $\left \lfloor
x\right \rfloor $ the largest integer $\leq x,$ see also \cite{Jac}. In 2012,
Cicho\'{n} and Golebiewski \cite{Cic} develop a method to give some other
applications and a generalization to multinomial distributions. The sequence
of polynomials $\left( \binom{n}{k}\left( 1-q\right) ^{k}q^{n-k}\right) $ is
the sequence of Bernstein polynomials which forms basis polynomials and used
to approximate functions based on Weierstrass theorem \cite{Ber}. In
algebraic combinatorics, to write the polynomial $S_{n}(q)$ in the basis $%
\left( q^{j};j\leq n\right) $ for particular cases of the sequence $\left(
a_{n}\right) ,$ several identities related to Bernstein polynomials exist in
the literature. For example, Mneimneh \cite{Mne} proves, on using a
probabilistic analysis, the following%
\begin{equation}
\overset{n}{\underset{k=0}{\sum }}H_{k}\binom{n}{k}\left( 1-q\right)
^{k}q^{n-k}=H_{n}-\overset{n}{\underset{j=.1}{\sum }}\frac{q^{j}}{j},
\label{a}
\end{equation}%
where $\left( H_{n};n\geq 0\right) $ are the harmonic numbers defined by 
\begin{equation*}
H_{0}=0\text{ \ and \ }H_{n}=\overset{n}{\underset{k=.1}{\sum }}\frac{1}{k}%
,\  \ n\geq 1.
\end{equation*}%
This identity appears as a special case of the Boyadzhiev's identity \cite%
{Boy} given as%
\begin{equation}
\overset{n}{\underset{k=0}{\sum }}H_{k}\binom{n}{k}x^{k}y^{n-k}=\left(
x+y\right) ^{n}H_{n}-\left( y\left( x+y\right) ^{n-1}+\frac{y^{2}}{2}\left(
x+y\right) ^{n-2}+\cdots +\frac{y^{n}}{n}\right) .  \label{b}
\end{equation}%
Also, in \cite{Kom}, Komatsu generalizes the identity (\ref{a}) by%
\begin{equation*}
\overset{n}{\underset{k=0}{\sum }}H_{k}^{\left( r\right) }\binom{n}{k}\left(
1-q\right) ^{k}q^{n-k}=H_{n}^{\left( r\right) }-\overset{n}{\underset{j=.1}{%
\sum }}D_{n}\left( r,j\right) \frac{q^{j}}{j},
\end{equation*}%
where%
\begin{equation*}
H_{0}^{\left( r\right) }=0\text{ \ and \ }H_{n}^{\left( r\right) }=\overset{n%
}{\underset{k=.1}{\sum }}\frac{1}{k^{r}},\  \ n\geq 1,\  \ r\geq 1,
\end{equation*}%
and%
\begin{equation*}
D_{n}\left( r,j\right) =\overset{j-1}{\underset{l=.0}{\sum }}\left(
-1\right) ^{j-l-1}\binom{n-l-1}{n-j}\binom{n}{j}\frac{1}{\left( n-l\right)
^{r-1}}.
\end{equation*}%
In this paper, we study the expression $S_{n}(q)=\overset{n}{\underset{k=0}{%
\sum }}a_{k}\binom{n}{k}\left( 1-q\right) ^{k}q^{n-k}$ for different
sequences $\left( a_{n}\right) .$ As it \ is shown above, this expression
has a best representation in the basis $\left( q^{j};j\leq n\right) $ in the
special cases $a_{n}=H_{n}$ and $a_{n}=H_{n}^{\left( r\right) }.$ In our
contribution, we give in the second section, the following representation of 
$S_{n}(q)$ in the basis $\left( q^{j};j\leq n\right) :$%
\begin{equation*}
S_{n}(q)=\overset{n}{\underset{j=0}{\sum }}\binom{n}{j}\left( -1\right)
^{j}\left( \nabla ^{j}a_{n}\right) q^{j}
\end{equation*}%
and we deduce best representations of $S_{n}(q)$ when $\left( a_{n}\right) $
is the generalized Harmonic numbers. We present also some applications of
the main proposition to the sequence of binomial coefficients, the sequence $%
\left( \left[ n\right] _{q}\right) ,$ Fibonacci and Lucas numbers, Meixner
and generalized Laguerre polynomials. In the third section, we give some
properties, relations related to the sequences $\left( S_{n}\left( x\right)
\right) $ and $\left( S_{n}\left( x+q-xq\right) \right) ,$ in which we prove
that the Bernoulli transform of the sequence $\left( S_{n}\left( x\right)
;n\geq 0\right) $ is the sequence $\left( S_{n}\left( x+q-xq\right) \right) ,
$ i.e. 
\begin{equation*}
S_{n}\left( x+q-xq\right) =\overset{n}{\underset{k=0}{\sum }}S_{k}\left(
x\right) \binom{n}{k}\left( 1-q\right) ^{k}q^{n-k}=\overset{n}{\underset{k=0}%
{\sum }}S_{k}\left( q\right) \binom{n}{k}\left( 1-x\right) ^{k}x^{n-k}
\end{equation*}%
with its general case%
\begin{eqnarray*}
\left( 1-q\right) ^{m}\overset{n}{\underset{k=0}{\sum }}S_{k+m}\left(
x\right) \binom{n}{k}\left( 1-q\right) ^{k}q^{n-k} &=&\left( 1-x\right) ^{m}%
\overset{n}{\underset{k=0}{\sum }}S_{k+m}\left( q\right) \binom{n}{k}\left(
1-x\right) ^{k}x^{n-k} \\
&=&\overset{m}{\underset{j=0}{\sum }}\binom{m}{j}\left( -q\right)
^{m-j}S_{j+n}\left( x+q-xq\right) ,
\end{eqnarray*}%
and we give their probabilistic interpretations. In the last section, we
give some application of the Bernoulli transform to Appell polynomials.

\section{A principal elementary identity and its applications}

Let $\left( a_{n};n\geq 0\right) $ be a sequence of real numbers and%
\begin{equation*}
b_{n}=a_{n}-a_{n-1}:=\nabla a_{n},\  \ n\geq 1.
\end{equation*}%
Let%
\begin{eqnarray}
S_{n}(q) &:&=\overset{n}{\underset{k=0}{\sum }}a_{k}\binom{n}{k}\left(
1-q\right) ^{k}q^{n-k}\text{ \ and \ }  \label{g} \\
\overline{S}_{n}(q) &:&=\overset{n}{\underset{k=0}{\sum }}a_{k}\binom{n}{k}%
\left( 1-q\right) ^{n-k}q^{k}=S_{n}(1-q).  \label{n2}
\end{eqnarray}

\begin{proposition}
\label{T1}For any real sequence $\left( a_{n};n\geq 0\right) $ there hold%
\begin{eqnarray}
S_{n}(q) &=&\overset{n}{\underset{j=0}{\sum }}\left( -1\right) ^{j}\binom{n}{%
j}M\left( n,j\right) q^{j}=\left( 1-q\nabla \right) ^{n}a_{n},  \label{c} \\
\overline{S}_{n}(q) &=&\overset{n}{\underset{j=0}{\sum }}\left( -1\right)
^{j}\binom{n}{j}\overline{M}\left( j\right) q^{j},  \label{cbis}
\end{eqnarray}%
where%
\begin{equation}
M\left( n,j\right) =\nabla ^{j}a_{n}=\overset{j}{\underset{l=0}{\sum }}%
\left( -1\right) ^{l}\binom{j}{l}a_{n-l}\text{ \ and \ }\overline{M}\left(
j\right) =\overset{j}{\underset{l=0}{\sum }}\left( -1\right) ^{l}\binom{j}{l}%
a_{l},\  \ j\geq 0.  \label{h}
\end{equation}
\end{proposition}

\begin{proof}
We prove that the left hand side of (\ref{c}) equal to its right hand side.
Indeed,%
\begin{eqnarray*}
\overset{n}{\underset{j=0}{\sum }}\left( -1\right) ^{j}\binom{n}{j}M\left(
n,j\right) q^{j} &=&\overset{n}{\underset{j=0}{\sum }}\left( -1\right) ^{j}%
\binom{n}{j}q^{j}\overset{j}{\underset{l=0}{\sum }}\left( -1\right) ^{l}%
\binom{j}{l}a_{n-l} \\
&=&\overset{n}{\underset{l=0}{\sum }}\left( -1\right) ^{l}a_{n-l}\overset{n}{%
\underset{j=l}{\sum }}\left( -1\right) ^{j}\binom{n}{j}\binom{j}{l}q^{j} \\
&=&\overset{n}{\underset{l=0}{\sum }}\left( -1\right) ^{n-l}a_{l}\overset{n}{%
\underset{j=n-l}{\sum }}\left( -1\right) ^{j}\binom{n}{j}\binom{j}{n-l}q^{j}
\\
&=&\overset{n}{\underset{l=0}{\sum }}\left( -1\right) ^{n-l}a_{l}\overset{l}{%
\underset{j=0}{\sum }}\left( -1\right) ^{j+n-l}\binom{n}{j+n-l}\binom{j+n-l}{%
n-l}q^{j+n-l} \\
&=&\overset{n}{\underset{l=0}{\sum }}a_{l}\binom{n}{l}q^{n-l}\overset{l}{%
\underset{j=0}{\sum }}\left( -1\right) ^{j}\binom{l}{j}q^{j} \\
&=&\overset{n}{\underset{l=0}{\sum }}a_{l}\binom{n}{l}q^{n-l}\left(
1-q\right) ^{l} \\
&=&S_{n}(q).
\end{eqnarray*}%
With the same way one can prove the second part of the proposition.
\end{proof}

\begin{remark}
One can observe easily that we have%
\begin{eqnarray}
\nabla M\left( n,k\right) &=&M\left( n,k\right) -M\left( n-1,k\right)
=M\left( n,k+1\right) ,\  \  \ k\geq 1,\ n\geq 1,  \label{i} \\
\nabla ^{j}M\left( n,k\right) &=&M\left( n,k+j\right) ,\  \ j\geq 0.
\label{j}
\end{eqnarray}%
Indeed, from definition, we get%
\begin{eqnarray*}
\nabla M\left( n,k\right) &=&M\left( n,k\right) -M\left( n-1,k\right) \\
&=&\nabla ^{k}a_{n}-\nabla ^{k}a_{n-1} \\
&=&\nabla ^{k}\left( a_{n}-a_{n-1}\right) \\
&=&\nabla ^{k+1}a_{n} \\
&=&M\left( n,k+1\right) ,\  \ k\geq 1.
\end{eqnarray*}
\end{remark}

\noindent Let $\left( H_{n}^{\left( r\right) }\left( x\right) \right) $ be
the sequence defined by%
\begin{equation*}
H_{0}^{\left( r\right) }\left( x\right) =0,\  \  \ H_{n}^{\left( r\right)
}\left( x\right) =\overset{n}{\underset{k=1}{\sum }}\frac{1}{\left(
k+x\right) ^{r}},\  \ n\geq 1,\ r\geq 1,
\end{equation*}%
As a consequence, for $a_{n}=H_{n}^{\left( r\right) }\left( x\right) $ of
Proposition \ref{T1}, we get%
\begin{equation}
\overset{n}{\underset{k=0}{\sum }}H_{k}^{\left( r\right) }\left( x\right) 
\binom{n}{k}\left( 1-q\right) ^{k}q^{n-k}=\overset{n}{\underset{j=0}{\sum }}%
\left( -1\right) ^{j}\binom{n}{j}M\left( n,j\right) q^{j},  \label{k}
\end{equation}%
where 
\begin{equation*}
M\left( n,j\right) =\overset{j}{\underset{l=0}{\sum }}\left( -1\right) ^{l}%
\binom{j}{l}\frac{1}{\left( n-l+x\right) ^{r}},\  \ n\geq 1.
\end{equation*}%
This last application can be deduced from the paper of Komatsu \cite{Kom}.

\begin{remark}
Since $\overline{S}_{n}(q)=S_{n}(1-q),$ then from Proposition \ref{T1} it
follows%
\begin{equation*}
\overset{n}{\underset{j=0}{\sum }}\left( -1\right) ^{j}\binom{n}{j}\overline{%
M}\left( j\right) q^{j}=\overset{n}{\underset{j=0}{\sum }}\left( -1\right)
^{j}\binom{n}{j}M\left( n,j\right) \left( 1-q\right) ^{j}.
\end{equation*}%
So, by identifying the coefficients of $q^{j}$ (respectively $\left(
1-q\right) ^{j}$) we obtain%
\begin{equation}
\overset{n}{\underset{j=0}{\sum }}\left( -1\right) ^{k-j}\binom{n}{j}M\left(
n+k,j+k\right) =\overline{M}\left( k\right) \text{ \ and \ }\overset{n}{%
\underset{j=0}{\sum }}\left( -1\right) ^{k-j}\binom{n}{j}\overline{M}\left(
j+k\right) =M\left( n+k,k\right) .  \label{i1}
\end{equation}
\end{remark}

For the applications of Proposition \ref{T1}, some special cases for the
sequence $\left( a_{n}\right) $ are given by the following corollaries.

\begin{corollary}
For any real numbers $a,b$ and $c$ we have%
\begin{equation}
\overset{n}{\underset{k=0}{\sum }}\mathcal{F}_{k+r}\binom{n}{k}\left(
1-q\right) ^{k}q^{n-k}=\sum_{j=0}^{n}\left( -1\right) ^{j}\binom{n}{j}%
\mathcal{F}_{n+r-2j}\left( cq\right) ^{j},  \label{e}
\end{equation}%
where $\left( \mathcal{F}_{n}\right) $ is the sequence of numbers defined by%
\begin{equation*}
\mathcal{F}_{n+2}=c\mathcal{F}_{n}+\mathcal{F}_{n+1},\  \  \  \mathcal{F}%
_{0}=a,\  \  \  \mathcal{F}_{1}=b,\  \mathcal{F}_{-n-1}=0,\  \ n\geq 0.
\end{equation*}%
In particular, we have%
\begin{eqnarray}
\overset{n}{\underset{k=0}{\sum }}F_{k+r}\binom{n}{k}\left( 1-q\right)
^{k}q^{n-k} &=&\sum_{j=0}^{n}\left( -1\right) ^{j}\binom{n}{j}%
F_{n+r-2j}q^{j},  \label{d} \\
\overset{n}{\underset{k=0}{\sum }}L_{k+r}\binom{n}{k}\left( 1-q\right)
^{k}q^{n-k} &=&\sum_{j=0}^{n}\left( -1\right) ^{j}\binom{n}{j}%
L_{n+r-2j}q^{j},  \label{dbis}
\end{eqnarray}%
where $\left( F_{n}\right) $ and $\left( L_{n}\right) $ are, respectively,
the Fibonacci and Lucas numbers.
\end{corollary}

\begin{proof}
For $a_{n}=\mathcal{F}_{n+r}$ in Proposition \ref{T1} we obtain%
\begin{equation*}
M\left( n,j\right) =\nabla ^{j}\mathcal{F}_{n+r}=c^{j}\mathcal{F}_{n+r-2j},\
\ j\geq 0.
\end{equation*}%
and, consequently, the desired identity follows.
\end{proof}

\begin{corollary}
For any real numbers $\alpha ,$ $\beta ,\ x$ and non-negative integer $m$ we
have%
\begin{eqnarray}
\overset{n}{\underset{k=0}{\sum }}L_{k+r}^{\left( \alpha \right) }\left(
x\right) \binom{n}{k}\left( 1-q\right) ^{k}q^{n-k} &=&\overset{n}{\underset{%
j=0}{\sum }}\left( -1\right) ^{j}\binom{n}{j}L_{n+r}^{\left( \alpha
-j\right) }\left( x\right) q^{j},  \label{l} \\
\overset{n}{\underset{k=0}{\sum }}L_{r}^{\left( \alpha +k\right) }\left(
x\right) \binom{n}{k}\left( 1-q\right) ^{k}q^{n-k} &=&\overset{n}{\underset{%
j=0}{\sum }}\left( -1\right) ^{j}\binom{n}{j}L_{r-j}^{\left( \alpha
+n\right) }\left( x\right) q^{j}  \label{m}
\end{eqnarray}%
and%
\begin{equation}
\overset{n}{\underset{k=0}{\sum }}M_{k+r}\left( x;\alpha ,\beta \right) 
\binom{n}{k}\left( 1-q\right) ^{k}q^{n-k}=\overset{n}{\underset{j=0}{\sum }}%
\left( -1\right) ^{j}\binom{n}{j}\frac{\left( \beta -1\right) ^{j}\left(
x+1\right) _{j}}{\left( \alpha -j-1\right) _{j}\beta ^{j}}M_{n+r-j}\left(
x;\alpha +j,\beta \right) q^{j},  \label{n}
\end{equation}%
where $\left( L_{n}^{\left( \alpha \right) }\left( x\right) \right) $ are
the generalized Laguerre polynomials and, $\left( M_{n}\left( x;\alpha
,\beta \right) \right) $ are the Meixner polynomials.
\end{corollary}

\begin{proof}
For $a_{n}=L_{n+r}^{\left( \alpha \right) }\left( x\right) $ in Proposition %
\ref{T1} we obtain \cite[Eq. 18]{Rob}%
\begin{equation*}
M\left( n,j\right) =\nabla ^{j}L_{n+r}^{\left( \alpha \right) }\left(
x\right) =L_{n+r}^{\left( \alpha -j\right) }\left( x\right) ,\  \ j\geq 0,
\end{equation*}%
for $a_{n}=L_{r}^{\left( \alpha +n\right) }\left( x\right) $ in Proposition %
\ref{T1} we obtain \cite[Eq. 21]{Rob}%
\begin{equation*}
M\left( n,j\right) =\nabla ^{j}L_{r}^{\left( \alpha +n\right) }\left(
x\right) =L_{r-j}^{\left( \alpha +n\right) }\left( x\right) ,\  \ j\geq 0,
\end{equation*}%
and for $a_{n}=M_{n+r}\left( x;\alpha ,\beta \right) $ in Proposition \ref%
{T1} we obtain \cite[Eq. 48]{Rob}%
\begin{equation*}
M\left( n,j\right) =\nabla ^{j}M_{n+r}\left( x;\alpha ,\beta \right) =\frac{%
\left( \beta -1\right) ^{j}\left( x+1\right) _{j}}{\left( \alpha -j-1\right)
_{j}\beta ^{j}}M_{n+r-j}\left( x;\alpha +j,\beta \right) ,\  \text{ }j\geq 0,
\end{equation*}%
and, consequently, the desired identity follows.
\end{proof}

\begin{corollary}
For any real number $\alpha $ and any integer $r\geq 0$ we have%
\begin{eqnarray}
\overset{n}{\underset{k=0}{\sum }}\left( -1\right) ^{k}\binom{\alpha }{k+r}%
\binom{n}{k}\left( 1-q\right) ^{k}q^{n-k} &=&\overset{n}{\underset{j=0}{\sum 
}}\left( -1\right) ^{n-j}\binom{n}{j}\binom{j+\alpha }{n+r}q^{j},  \label{p}
\\
\overset{n}{\underset{k=0}{\sum }}\left( -1\right) ^{k}\binom{\alpha }{k+r}%
\binom{n}{k}\left( 1-q\right) ^{n-k}q^{k} &=&\overset{n}{\underset{j=0}{\sum 
}}\left( -1\right) ^{j}\binom{n}{j}\binom{j+\alpha }{j+r}q^{j}.  \label{pbis}
\end{eqnarray}%
In particular, for $\alpha =n$ and $r=0,$ we get%
\begin{eqnarray}
\overset{n}{\underset{k=0}{\sum }}\left( -1\right) ^{k}\binom{n}{k}%
^{2}\left( 1-q\right) ^{k}q^{n-k} &=&\overset{n}{\underset{j=0}{\sum }}%
\left( -1\right) ^{n-j}\binom{n}{j}\binom{j+n}{n}q^{j},  \label{q} \\
\overset{n}{\underset{k=0}{\sum }}\left( -1\right) ^{k}\binom{n}{k}%
^{2}\left( 1-q\right) ^{n-k}q^{k} &=&\overset{n}{\underset{j=0}{\sum }}%
\left( -1\right) ^{j}\binom{n}{j}\binom{j+n}{n}q^{j}.  \label{qbis}
\end{eqnarray}%
We also have%
\begin{equation}
\overset{n}{\underset{j=0}{\sum }}\left( -1\right) ^{n-j}\binom{n}{j}\binom{%
j+k+\alpha }{n+k+r}=\binom{k+\alpha }{k+r}\text{ \ and \ }\overset{n}{%
\underset{j=0}{\sum }}\left( -1\right) ^{n-j}\binom{n}{j}\binom{j+k+\alpha }{%
j+k+r}=\binom{k+\alpha }{n+k+r}.  \label{j1}
\end{equation}
\end{corollary}

\begin{proof}
For $a_{n}=\left( -1\right) ^{n}\binom{\alpha }{n+r}$ in Proposition \ref{T1}
we we get%
\begin{eqnarray*}
M\left( n,j\right) &=&\nabla ^{j}a_{n}=\overset{j}{\underset{l=0}{\sum }}%
\left( -1\right) ^{l}\binom{j}{l}a_{n-l}=\left( -1\right) ^{n}\overset{j}{%
\underset{l=0}{\sum }}\binom{j}{l}\binom{\alpha }{r+n-l}=\left( -1\right)
^{n}\binom{j+\alpha }{n+r}, \\
\overline{M}\left( j\right) &=&\overset{j}{\underset{l=0}{\sum }}\left(
-1\right) ^{l}\binom{j}{l}a_{l}=\overset{j}{\underset{l=0}{\sum }}\binom{j}{l%
}\binom{\alpha }{r+l}=\overset{j}{\underset{l=0}{\sum }}\binom{j}{j-l}\binom{%
\alpha }{r+l}=\binom{j+\alpha }{j+r},
\end{eqnarray*}%
and, consequently, the identities (\ref{q}) and (\ref{qbis}) follow. The
last identity follows from (\ref{i1}).
\end{proof}

\begin{corollary}
For $m\geq n\geq 0$ and $s\geq 0$ we have 
\begin{equation}
\overset{n}{\underset{k=0}{\sum }}\frac{k}{sm+k}\binom{sm+k}{m}\binom{n}{k}%
\left( 1-q\right) ^{k}q^{n-k}=\overset{n}{\underset{j=0}{\sum }}\left(
-1\right) ^{j}\binom{n}{j}\frac{n}{s\left( m-j\right) +n}\binom{sm+n-j}{m-j}%
q^{j},\  \ m\geq n.  \label{b3}
\end{equation}
\end{corollary}

\begin{proof}
For $a_{n}:=A_{m}(s,n)=\frac{n}{sm+n}\binom{sm+n}{m}$ in Proposition \ref{T1}
and from the recurrence relation \cite{Goy}:%
\begin{equation*}
A_{m}(s,n)=A_{m}(s,n-1)+A_{m-1}(s,s+n-1),\  \  \ m\geq 1,
\end{equation*}%
we get $M\left( n,0\right) =A_{m}(s,n),\  \ M\left( n,1\right) =\nabla
a_{n}=A_{m-1}(s,s+n-1),$ i.e.%
\begin{equation*}
M\left( n,j\right) =\nabla ^{j}a_{n}=A_{m-j}(s,j\left( s-1\right) +n),
\end{equation*}%
and, consequently, the desired identity follows.
\end{proof}

\begin{corollary}
For any integer $r\geq 0$ we have%
\begin{equation}
\sum_{k=0}^{n}\left[ k+r\right] _{p}\binom{n}{k}(1-q)^{k}q^{\,n-k}=\frac{1}{%
1-p}\left( 1-p^{\,r}\left( p+q(1-p)\right) ^{n}\right) ,\  \ p\neq 1,
\label{r}
\end{equation}%
where%
\begin{equation*}
\left[ n\right] _{p}:=1+p+\cdots +p^{n-1}=\left( 1-p^{n}\right) /\left(
1-p\right) \text{ if }p\neq 1\text{ and }\left[ n\right] _{1}=n.
\end{equation*}
\end{corollary}

\begin{proof}
For $a_{n}=\left[ n+r\right] _{p}$ in Proposition \ref{T1} we obtain%
\begin{eqnarray*}
M\left( n,0\right) &=&\left[ n+r\right] _{p}, \\
M\left( n,j\right) &=&\Delta ^{j}\left[ n+r\right] _{p}=p^{%
\,n+r-j}(1-p)^{j-1},\  \text{ }j\geq 1,\ p\neq 1,
\end{eqnarray*}%
and%
\begin{eqnarray*}
\sum_{k=0}^{n}\left[ k+r\right] _{p}\binom{n}{k}(1-q)^{k}q^{\,n-k} &=&\left[
n+r\right] _{p}+\overset{n}{\underset{j=1}{\sum }}\left( -1\right) ^{j}%
\binom{n}{j}p^{\,n+r-j}(1-p)^{j-1}q^{j} \\
&=&\frac{1-p^{n+r}}{1-p}+\frac{p^{\,n+r}}{1-p}\overset{n}{\underset{j=1}{%
\sum }}\binom{n}{j}\left( \frac{q(1-p)}{p}\right) ^{j} \\
&=&\frac{1}{1-p}\left( 1-p^{\,r}\left( p+q(1-p)\right) ^{n}\right) ,
\end{eqnarray*}%
and, consequently, the desired identity follows.
\end{proof}

\begin{corollary}
Let $\alpha \neq 0$ be a real number. We have%
\begin{eqnarray}
x\overset{n}{\underset{k=0}{\sum }}\left( -1\right) ^{k}\mathcal{B}%
_{k}\left( x\right) \binom{n}{k}\left( 1-q\right) ^{n-k}q^{k} &=&\overset{n}{%
\underset{j=0}{\sum }}\left( -1\right) ^{j}\binom{n}{j}\mathcal{B}%
_{j+1}\left( x\right) q^{j},  \label{l1} \\
x\overset{n}{\underset{k=0}{\sum }}\left( -1\right) ^{k}w_{k}\left( x\right) 
\binom{n}{k}\left( 1-q\right) ^{n-k}q^{k} &=&\left( x+1\right) \overset{n}{%
\underset{j=0}{\sum }}\binom{n}{j}w_{j}\left( x\right) \left( -q\right)
^{j}-1,  \label{p1}
\end{eqnarray}%
where $\mathcal{B}_{n}\left( x\right) $ and $w_{n}\left( x\right) $ are,
respectively, the $n$-th single variable Bell and the geometric polynomials,
i.e.%
\begin{equation*}
\mathcal{B}_{n}\left( x\right) =\overset{n}{\underset{j=0}{\sum }}\QATOPD \{
\} {n}{j}x^{j}\text{ \ and \ }w_{n}\left( x\right) =\overset{n}{\underset{j=0%
}{\sum }}j!\QATOPD \{ \} {n}{j}x^{j}
\end{equation*}%
with $\QATOPD \{ \} {n}{j}$ is the $\left( n,j\right) $-th Stirling number
of second kind \cite{Com}.
\end{corollary}

\begin{proof}
For $a_{k}=\left( -1\right) ^{k}x\mathcal{B}_{k}\left( x\right) $ in
Proposition \ref{T1} we obtain%
\begin{equation*}
\overline{M}\left( j\right) =x\overset{j}{\underset{l=0}{\sum }}\left(
-1\right) ^{l}\binom{j}{l}\left( -1\right) ^{l}\mathcal{B}_{l}\left(
x\right) =x\overset{j}{\underset{l=0}{\sum }}\binom{j}{l}\mathcal{B}%
_{l}\left( x\right)
\end{equation*}%
and since $x\overset{j}{\underset{l=0}{\sum }}\binom{j}{l}\mathcal{B}%
_{l}\left( x\right) =\mathcal{B}_{j+1}\left( x\right) $ \cite{Boy1,Dil}, we
get $\overline{M}\left( j\right) =\mathcal{B}_{j+1}\left( x\right) \ $and%
\begin{equation*}
\overline{S}_{n}(q)=x\overset{n}{\underset{k=0}{\sum }}\left( -1\right) ^{k}%
\mathcal{B}_{k}\left( x\right) \binom{n}{k}\left( 1-q\right) ^{n-k}q^{k}=%
\overset{n}{\underset{j=0}{\sum }}\left( -1\right) ^{j}\binom{n}{j}\overline{%
M}\left( j\right) q^{j}=\overset{n}{\underset{j=0}{\sum }}\binom{n}{j}%
\mathcal{B}_{j+1}\left( x\right) \left( -q\right) ^{j},
\end{equation*}%
and, consequently, the first identity follows.\newline
Similarly, for $a_{k}=\left( -1\right) ^{k}xw_{k}\left( x\right) $ in
Proposition \ref{T1} we obtain%
\begin{equation*}
\overline{M}\left( j\right) =x\overset{j}{\underset{l=0}{\sum }}\left(
-1\right) ^{l}\binom{j}{l}\left( -1\right) ^{l}w_{l}\left( x\right) =x%
\overset{j}{\underset{l=0}{\sum }}\binom{j}{l}w_{l}\left( x\right)
\end{equation*}%
and since for $n\geq 1$ there holds $\left( x+1\right) w_{n}\left( x\right)
=x\sum_{k=0}^{n}\binom{n}{k}w_{k}\left( x\right) $ \cite{Dil,Mih3}, we get%
\begin{equation*}
\overline{M}\left( 0\right) =x\text{ and }\overline{M}\left( j\right)
=\left( x+1\right) w_{j}\left( x\right) ,\text{ }j\geq 1,
\end{equation*}%
and%
\begin{eqnarray*}
\overline{S}_{n}(q) &=&x\overset{n}{\underset{k=0}{\sum }}\left( -1\right)
^{k}w_{k}\left( x\right) \binom{n}{k}\left( 1-q\right) ^{n-k}q^{k} \\
&=&M\left( 0\right) +\overset{n}{\underset{j=1}{\sum }}\left( -1\right) ^{j}%
\binom{n}{j}\overline{M}\left( j\right) q^{j} \\
&=&x+\left( x+1\right) \overset{n}{\underset{j=1}{\sum }}\binom{n}{j}%
w_{j}\left( x\right) \left( -q\right) ^{j} \\
&=&\left( x+1\right) \overset{n}{\underset{j=0}{\sum }}\binom{n}{j}%
w_{j}\left( x\right) \left( -q\right) ^{j}-1.
\end{eqnarray*}
\end{proof}

\begin{remark}
For $q\rightarrow 1$ in the above corollaries it follow that we have%
\begin{eqnarray*}
\overset{n}{\underset{j=0}{\sum }}\left( -1\right) ^{j}\binom{n}{j}\mathcal{F%
}_{n+r-2j}c^{j} &=&\mathcal{F}_{r}, \\
\overset{n}{\underset{j=0}{\sum }}\left( -1\right) ^{n-j}\binom{n}{j}\binom{%
\alpha +j}{r+n} &=&\binom{\alpha }{r}, \\
\overset{n}{\underset{j=0}{\sum }}\left( -1\right) ^{j}\binom{n}{j}\frac{n}{%
s\left( m-j\right) +n}\binom{sm+n-j}{m-j} &=&0,\  \ m\geq n. \\
\overset{n}{\underset{j=0}{\sum }}\left( -1\right) ^{j}\binom{n}{j}%
L_{n+r}^{\left( \alpha -j\right) }\left( x\right) &=&L_{r}^{\left( \alpha
\right) }\left( x\right) , \\
\overset{n}{\underset{j=0}{\sum }}\left( -1\right) ^{j}\binom{n}{j}%
L_{r-j}^{\left( \alpha +n\right) }\left( x\right) &=&L_{r}^{\left( \alpha
\right) }\left( x\right) , \\
\overset{n}{\underset{j=0}{\sum }}\left( -1\right) ^{j}\binom{n}{j}\frac{%
\left( \beta -1\right) ^{j}\left( x+1\right) _{j}}{\left( \alpha -j-1\right)
_{j}\beta ^{j}}M_{n+r-j}\left( x;\alpha +j,\beta \right) &=&M_{r}\left(
x;\alpha ,\beta \right) , \\
\overset{n}{\underset{j=0}{\sum }}\left( -1\right) ^{n-j}\binom{n}{j}%
\mathcal{B}_{j+1}\left( x\right) &=&x\mathcal{B}_{n}\left( x\right) , \\
\left( x+1\right) \overset{n}{\underset{j=0}{\sum }}\left( -1\right) ^{n-j}%
\binom{n}{j}w_{j}\left( x\right) &=&xw_{n}\left( x\right) +\left( -1\right)
^{n}.
\end{eqnarray*}
\end{remark}

\section{Some properties of the sequences $\left( S_{n}(q)\right) $}

In this section, we give some properties related to the sequences $\left(
S_{n}\left( x\right) ;n\geq 0\right) $ and $\left( S_{n}\left( x+q-xq\right)
;n\geq 0\right) .$

\begin{proposition}
\label{P1}Let $A(z)=\sum_{n\geq 0}a_{n}z^{n}$. There holds%
\begin{equation}
S_{n}(q)=\left[ z^{n}\right] \ (1-q+qz)^{n}A(z)=\left[ z^{n}\right] \  \frac{1%
}{1-qz}A\left( \frac{\left( 1-q\right) z}{1-qz}\right) .  \label{s}
\end{equation}
\end{proposition}

\begin{proof}
From the expansion%
\begin{equation*}
(1-q+qz)^{n}=\sum_{k=0}^{n}\binom{n}{k}q^{\,n-k}(1-q)^{k}z^{k}
\end{equation*}%
the coefficient of $z^{n}$ in the product $(1-q+qz)^{n}A(z)$ is given by%
\begin{align*}
\lbrack z^{n}](1-q+qz)^{n}A(z)& =[z^{n}]\  \left( \sum_{i=0}^{n}\binom{n}{i}%
q^{\,i}(1-q)^{n-i}z^{i}\right) \left( \sum_{j\geq 0}a_{j}z^{j}\right) \\
& =[z^{n}]\  \sum_{k\geq 0}z^{k}\underset{i=0}{\overset{\min \left(
n,k\right) }{\sum }}a_{k-i}\binom{n}{i}q^{\,i}(1-q)^{n-i} \\
& =\sum_{k=0}^{n}a_{n-k}\binom{n}{k}q^{\,k}(1-q)^{n-k} \\
& =\sum_{k=0}^{n}a_{k}\binom{n}{k}q^{\,n-k}(1-q)^{k} \\
& =S_{n}(q),
\end{align*}%
and this completes the proof of the first identity. The second\ identity is
known \cite{Cic,Fla,Gou,Hau}, and we present here a new proof based on
Lagrange inversion theorem \cite{Com,Ges,Sur} and on the last identity $%
S_{n}(q)=\left[ z^{n}\right] \ (1-q+qz)^{n}A(z)$. Indeed, the equation $%
z=t\left( 1-q+qz\right) $ has solution $z=\frac{\left( 1-q\right) t}{1-qt}.$
For any analytic function about zero $F$ it is known that we have%
\begin{equation*}
F\left( z\right) =F\left( 0\right) +\underset{n\geq 1}{\sum }\frac{t^{n}}{n!}%
D_{u=0}^{n-1}\left( F^{\prime }\left( u\right) \left( 1-q+qu\right)
^{n}\right) ,
\end{equation*}%
which gives%
\begin{equation*}
z^{\prime }F^{\prime }\left( z\right) =\underset{n\geq 0}{\sum }\frac{t^{n}}{%
n!}D_{u=0}^{n}\left( F^{\prime }\left( u\right) \left( 1-q+qu\right)
^{n+1}\right)
\end{equation*}%
or equivalently%
\begin{equation*}
\frac{1-q}{\left( 1-qt\right) ^{2}}F^{\prime }\left( \frac{\left( 1-q\right)
t}{1-qt}\right) =\underset{n\geq 0}{\sum }\frac{t^{n}}{n!}D_{u=0}^{n}\left(
\left( 1-q+qu\right) F^{\prime }\left( u\right) \left( 1-q+qu\right)
^{n}\right) .
\end{equation*}%
Then when one choice $F$ such that $\left( 1-q+qt\right) F^{\prime }\left(
t\right) =A\left( t\right) ,$ he gives%
\begin{equation*}
F^{\prime }\left( \frac{\left( 1-q\right) t}{1-qt}\right) =\frac{1-qt}{1-q}%
A\left( \frac{\left( 1-q\right) t}{1-qt}\right) .
\end{equation*}%
Hence, by the identity%
\begin{equation*}
S_{n}(q)=\left[ z^{n}\right] \ (1-q+qz)^{n}A(z)=\frac{1}{n!}%
D_{u=0}^{n}\left( \left( 1-q+qu\right) ^{n}A\left( u\right) \right)
\end{equation*}%
it follows%
\begin{equation*}
\frac{1}{1-qt}A\left( \frac{\left( 1-q\right) t}{1-qt}\right) =\underset{%
n\geq 0}{\sum }\frac{t^{n}}{n!}D_{u=0}^{n}\left( A\left( u\right) \left(
1-q+qu\right) ^{n}\right) =\underset{n\geq 0}{\sum }S_{n}\left( q\right)
t^{n},
\end{equation*}%
i.e. $S_{n}\left( q\right) =\left[ t^{n}\right] \  \frac{1}{1-qt}A\left( 
\frac{\left( 1-q\right) t}{1-qt}\right) .$
\end{proof}

\begin{proposition}
\label{P2}If the sequence $\left( S_{n}\left( q\right) \right) $ is the
Bernoulli transform of the sequence $\left( a_{n}\right) ,$ then the
sequence $\left( S_{n}(x+q-xq)\right) $ is the Bernoulli transform of the
sequence $\left( S_{n}\left( x\right) \right) .$ \newline
In other words, there holds%
\begin{equation}
S_{n}(x+q-xq)=S_{n}\left( 1-\left( 1-x\right) \left( 1-q\right) \right) =%
\overset{n}{\underset{k=0}{\sum }}S_{k}\left( x\right) \binom{n}{k}\left(
1-q\right) ^{k}q^{n-k},  \label{d1}
\end{equation}%
or equivalently,%
\begin{equation}
S_{n}\left( 1-xq\right) =\overset{n}{\underset{k=0}{\sum }}S_{k}\left(
1-x\right) \binom{n}{k}\left( 1-q\right) ^{n-k}q^{k}.  \label{d1bis}
\end{equation}
\end{proposition}

\begin{proof}
From the definition of $S_{n}\left( x\right) $ we get%
\begin{eqnarray*}
\overset{n}{\underset{k=0}{\sum }}S_{k}\left( x\right) \binom{n}{k}\left(
1-q\right) ^{k}q^{n-k} &=&\overset{n}{\underset{k=0}{\sum }}\left( \overset{k%
}{\underset{j=0}{\sum }}a_{j}\binom{k}{j}\left( 1-x\right)
^{j}x^{k-j}\right) \binom{n}{k}\left( 1-q\right) ^{k}q^{n-k} \\
&=&\overset{n}{\underset{j=0}{\sum }}a_{j}\overset{n}{\underset{k=j}{\sum }}%
\binom{n}{k}\binom{k}{j}\left( 1-x\right) ^{j}x^{k-j}\left( 1-q\right)
^{k}q^{n-k}
\end{eqnarray*}%
and since $\binom{n}{k}\binom{k}{j}=\binom{n}{j}\binom{n-j}{k-j}$ we can
write%
\begin{eqnarray*}
\overset{n}{\underset{k=0}{\sum }}S_{k}\left( x\right) \binom{n}{k}\left(
1-q\right) ^{k}q^{n-k} &=&\overset{n}{\underset{j=0}{\sum }}a_{j}\binom{n}{j}%
\overset{n}{\underset{k=j}{\sum }}\binom{n-j}{k-j}\left( 1-x\right)
^{j}x^{k-j}\left( 1-q\right) ^{k}q^{n-k} \\
&=&\overset{n}{\underset{j=0}{\sum }}a_{j}\binom{n}{j}\left( 1-x\right)
^{j}\left( 1-q\right) ^{j}\overset{n-j}{\underset{l=0}{\sum }}\binom{n-j}{l}%
\left( x\left( 1-q\right) \right) ^{l}q^{n-j-l} \\
&=&\overset{n}{\underset{j=0}{\sum }}a_{j}\binom{n}{j}\left( \left(
1-x\right) \left( 1-q\right) \right) ^{j}\left( x\left( 1-q\right) +q\right)
^{n-j} \\
&=&S_{n}(x+q-xq).
\end{eqnarray*}
\end{proof}

\noindent For the particular cases $x+q-xq=q^{m}$ or $x+q-xq=\left( \alpha
+1\right) q$ we get:

\begin{corollary}
\label{C1}For any positive integer $m$ we have%
\begin{eqnarray}
S_{n}(q^{m}) &=&\overset{n}{\underset{k=0}{\sum }}S_{k}\left( 1-\left[ q%
\right] _{m}\right) \binom{n}{k}\left( 1-q\right) ^{k}q^{n-k},  \label{e1} \\
S_{n}(q^{m}) &=&\overset{n}{\underset{k=0}{\sum }}S_{k}\left( q\right) 
\binom{n}{k}\left( \left[ q\right] _{m}\right) ^{k}\left( 1-\left[ q\right]
_{m}\right) ^{n-k}  \label{f1}
\end{eqnarray}%
which give when $q\in \left] 0,1\right[ $ and $m\longrightarrow \infty :$%
\begin{equation}
a_{n}=\overset{n}{\underset{k=0}{\sum }}S_{k}\left( -\frac{q}{1-q}\right) 
\binom{n}{k}\left( 1-q\right) ^{k}q^{n-k}=\frac{1}{\left( 1-q\right) ^{n}}%
\overset{n}{\underset{k=0}{\sum }}S_{k}\left( q\right) \binom{n}{k}\left(
-q\right) ^{n-k}.  \label{q2}
\end{equation}%
Also, for any real number $\alpha ,$ we have%
\begin{eqnarray}
S_{n}(\left( \alpha +1\right) q) &=&\overset{n}{\underset{k=0}{\sum }}%
S_{k}\left( \frac{\alpha q}{1-q}\right) \binom{n}{k}\left( 1-q\right)
^{k}q^{n-k},  \label{g1} \\
S_{n}(\left( \alpha +1\right) q) &=&\frac{1}{\left( 1-q\right) ^{n}}\overset{%
n}{\underset{k=0}{\sum }}S_{k}\left( q\right) \binom{n}{k}\left( 1-\left(
\alpha +1\right) q\right) ^{k}\left( \alpha q\right) ^{n-k}.  \label{h1}
\end{eqnarray}
\end{corollary}

\noindent More generally of the identity (\ref{d1}), the following
proposition gives the Bernoulli transform of the sequence $\left( \left(
1-q\right) ^{m}S_{n+m}\left( x\right) \right) ,$ i.e.\ it express the sum%
\begin{equation*}
\overset{n}{\underset{k=0}{\sum }}S_{k+m}\left( x\right) \binom{n}{k}\left(
1-q\right) ^{k}q^{n-k}
\end{equation*}%
by $S_{n}(x+q-xq),$ $S_{n+1}(x+q-xq),$ $\ldots ,$ $S_{n+m}(x+q-xq).$

\begin{proposition}
\label{P3}Let $m$ be a nonnegative integer. If the sequence $\left(
S_{n}\left( q\right) \right) $ is the Bernoulli transform of the sequence $%
\left( a_{n}\right) ,$ then the sequence 
\begin{equation*}
\left( \overset{m}{\underset{j=0}{\sum }}\binom{m}{j}\left( -q\right)
^{m-j}S_{j+n}\left( x+q-xq\right) ;n\geq 0\right)
\end{equation*}%
is the Bernoulli transform of the sequence $\left( \left( 1-q\right)
^{m}S_{n+m}\left( x\right) \right) .$ \newline
In other words, there holds%
\begin{equation}
\left( 1-q\right) ^{m}\overset{n}{\underset{k=0}{\sum }}S_{k+m}\left(
x\right) \binom{n}{k}\left( 1-q\right) ^{k}q^{n-k}=\overset{m}{\underset{j=0}%
{\sum }}\binom{m}{j}\left( -q\right) ^{m-j}S_{j+n}\left( x+q-xq\right) .
\label{q1}
\end{equation}%
In particular, for $m=0,$ $m=1$ and $m=2,$ we have%
\begin{eqnarray}
\overset{n}{\underset{k=0}{\sum }}S_{k}\left( x\right) \binom{n}{k}\left(
1-q\right) ^{k}q^{n-k} &=&S_{n}\left( x+q-xq\right) ,  \label{r1} \\
\left( 1-q\right) \overset{n}{\underset{k=0}{\sum }}S_{k+1}\left( x\right) 
\binom{n}{k}\left( 1-q\right) ^{k}q^{n-k} &=&S_{n+1}\left( x+q-xq\right)
-qS_{n}\left( x+q-xq\right) ,  \label{s1} \\
\left( 1-q\right) ^{2}\overset{n}{\underset{k=0}{\sum }}S_{k+2}\left(
x\right) \binom{n}{k}\left( 1-q\right) ^{k}q^{n-k} &=&S_{n+2}\left(
x+q-xq\right) -2qS_{n+1}\left( x+q-xq\right)  \label{t1} \\
&&+q^{2}S_{n}\left( x+q-xq\right) ,  \notag
\end{eqnarray}%
and for $n=0$ we obtain the inverse relation of (\ref{d1}):%
\begin{equation}
\left( 1-q\right) ^{m}S_{m}\left( x\right) =\overset{m}{\underset{j=0}{\sum }%
}\binom{m}{j}\left( -q\right) ^{m-j}S_{j}\left( x+q-xq\right) .  \label{i2}
\end{equation}
\end{proposition}

\begin{proof}
We proceed by induction on $m.$ The identity is true for $m=0$ and for all $%
n\geq 0$ because in this case is exactly the identity (\ref{d1}). Assume
that is true for $m$ for all $n\geq 0.$ We can write%
\begin{eqnarray*}
&&\overset{m}{\underset{j=0}{\sum }}\binom{m}{j}\left( -q\right)
^{m-j}S_{j+n+1}\left( x+q-xq\right) \\
&=&\left( 1-q\right) ^{m}\overset{n+1}{\underset{k=0}{\sum }}S_{k+m}\left(
x\right) \binom{n+1}{k}\left( 1-q\right) ^{k}q^{n+1-k} \\
&=&\left( 1-q\right) ^{m}\overset{n+1}{\underset{k=1}{\sum }}S_{k+m}\left(
x\right) \binom{n}{k-1}\left( 1-q\right) ^{k}q^{n+1-k}+q\left( 1-q\right)
^{m}\overset{n}{\underset{k=0}{\sum }}S_{k+m}\left( x\right) \binom{n}{k}%
\left( 1-q\right) ^{k}q^{n-k} \\
&=&\left( 1-q\right) ^{m+1}\overset{n}{\underset{k=0}{\sum }}S_{k+m+1}\left(
x\right) \binom{n}{k}\left( 1-q\right) ^{k}q^{n-k}+q\left( 1-q\right) ^{m}%
\overset{n}{\underset{k=0}{\sum }}S_{k+m}\left( x\right) \binom{n}{k}\left(
1-q\right) ^{k}q^{n-k} \\
&=&\left( 1-q\right) ^{m+1}\overset{n}{\underset{k=0}{\sum }}S_{k+m+1}\left(
x\right) \binom{n}{k}\left( 1-q\right) ^{k}q^{n-k}+q\overset{m}{\underset{j=0%
}{\sum }}\binom{m}{j}\left( -q\right) ^{m-j}S_{j+n}\left( x+q-xq\right) ,
\end{eqnarray*}%
hence, we get%
\begin{eqnarray*}
&&\left( 1-q\right) ^{m+1}\overset{n}{\underset{k=0}{\sum }}S_{k+m+1}\left(
x\right) \binom{n}{k}\left( 1-q\right) ^{k}q^{n-k} \\
&=&\overset{m}{\underset{j=0}{\sum }}\binom{m}{j}\left( -q\right)
^{m-j}S_{j+n+1}\left( x+q-xq\right) -q\overset{m}{\underset{j=0}{\sum }}%
\binom{m}{j}\left( -q\right) ^{m-j}S_{j+n}\left( x+q-xq\right) \\
&=&\overset{m+1}{\underset{j=1}{\sum }}\binom{m}{j-1}\left( -q\right)
^{m+1-j}S_{j+n}\left( x+q-xq\right) +\overset{m}{\underset{j=0}{\sum }}%
\binom{m}{j}\left( -q\right) ^{m+1-j}S_{j+n}\left( x+q-xq\right) \\
&=&\overset{m+1}{\underset{j=0}{\sum }}\left( \binom{m}{j-1}+\binom{m}{j}%
\right) \left( -q\right) ^{m+1-j}S_{j+n}\left( x+q-xq\right) \\
&=&\overset{m+1}{\underset{j=0}{\sum }}\binom{m+1}{j}\left( -q\right)
^{m+1-j}S_{j+n}\left( x+q-xq\right) ,
\end{eqnarray*}%
and this completes the induction step.
\end{proof}

\noindent Therefore, for $x=0$ in Proposition \ref{P3} we get $S_{n}\left(
0\right) =a_{n}$ and we obtain the following corollary.

\begin{corollary}
\label{C2}Let $m$ be a nonnegative integer. If the sequence $\left(
S_{n}\left( q\right) \right) $ is the Bernoulli transform of the sequence $%
\left( a_{n}\right) ,$ then the sequence 
\begin{equation*}
\left( \overset{m}{\underset{j=0}{\sum }}\binom{m}{j}\left( -q\right)
^{m-j}S_{j+n}\left( q\right) ;n\geq 0\right)
\end{equation*}%
is the Bernoulli transform of the sequence $\left( \left( 1-q\right)
^{m}a_{n+m};n\geq 0\right) .$ \newline
In other words, there holds%
\begin{equation}
\left( 1-q\right) ^{m}\overset{n}{\underset{k=0}{\sum }}a_{k+m}\binom{n}{k}%
\left( 1-q\right) ^{k}q^{n-k}=\overset{m}{\underset{j=0}{\sum }}\binom{m}{j}%
\left( -q\right) ^{m-j}S_{j+n}\left( q\right) .  \label{u1}
\end{equation}%
In particular, for $m=0,$ $m=1$ and $m=2,$ we have%
\begin{eqnarray}
\overset{n}{\underset{k=0}{\sum }}a_{k}\binom{n}{k}\left( 1-q\right)
^{k}q^{n-k} &=&S_{n}\left( q\right) ,  \label{v1} \\
\left( 1-q\right) \overset{n}{\underset{k=0}{\sum }}a_{k+1}\binom{n}{k}%
\left( 1-q\right) ^{k}q^{n-k} &=&S_{n+1}\left( q\right) -qS_{n}\left(
q\right) ,  \label{w1} \\
\left( 1-q\right) ^{2}\overset{n}{\underset{k=0}{\sum }}a_{k+2}\binom{n}{k}%
\left( 1-q\right) ^{k}q^{n-k} &=&S_{n+2}\left( q\right) -2qS_{n+1}\left(
q\right) +q^{2}S_{n}\left( q\right) ,  \label{x1}
\end{eqnarray}%
and for $n=0$ we find the identity (\ref{q2}) which is the inverse relation
of (\ref{v1}):%
\begin{equation}
\left( 1-q\right) ^{m}a_{m}=\overset{m}{\underset{j=0}{\sum }}\binom{m}{j}%
\left( -q\right) ^{m-j}S_{j}\left( q\right) .  \label{j2}
\end{equation}
\end{corollary}

\begin{remark}
- A similar transform of (\ref{w1}) is given in Theorem 1 of the paper \cite%
{Spi} for the Bernoulli transform when $q=1/2.$ The identity (\ref{j2})
represents also a recurrence relation for the sequence $\left( S_{m}\left(
q\right) \right) .$
\end{remark}

\noindent To give a probabilistic interpretation of Proposition \ref{P3},
let $\left( X_{n};n\geq 1\right) $ and $\left( Y_{n};n\geq 1\right) $ be two
independent sequences of independent random variables having the same law of
Bernoulli with parameters $1-q$ and $1-x$ respectively, i.e.%
\begin{eqnarray*}
P\left( X_{n}=1\right) &=&1-x\text{ \ and \ }P\left( X_{n}=0\right) =x,\  \
n\geq 1,\text{ and,} \\
P\left( Y_{n}=1\right) &=&1-y\text{ \ and \ }P\left( Y_{n}=0\right) =y,\  \
n\geq 1.
\end{eqnarray*}%
Setting $Z\left( n\right) =X_{1}+\cdots +X_{n},$ $T\left( n\right)
=Y_{1}+\cdots +Y_{n},$ $n\geq 1$ and $Z\left( 0\right) =T\left( 0\right) =0.$
It is known that $Z\left( n\right) $ and $T\left( n\right) $ have binomial
distributions with parameters $\left( n,1-x\right) $ and $\left(
n,1-y\right) $ respectively, i.e.%
\begin{equation*}
P\left( Z\left( n\right) =k\right) =\binom{n}{k}\left( 1-x\right)
^{k}x^{n-k},\text{ \ }P\left( T\left( n\right) =k\right) =\binom{n}{k}\left(
1-y\right) ^{k}y^{n-k},\text{ \ }k\in \left \{ 0,1,\ldots ,n\right \} .
\end{equation*}

\begin{corollary}[Probabilistic interpretation]
Let%
\begin{equation}
W\left( n\right) =T\left( Z\left( n\right) \right) =T\circ Z\left( n\right)
=Y_{1}+\cdots +Y_{X_{1}+\cdots +X_{n}}.  \label{k2}
\end{equation}%
Then, $W\left( n\right) $ has a binomial distribution with parameters $%
\left( n,\left( 1-x\right) \left( 1-y\right) \right) .$ \newline
By setting $a_{k}=f\left( k\right) ,$ the identity (\ref{q1}) can be written
as%
\begin{equation}
\left( 1-x\right) ^{m}\limfunc{E}f\left( T\left( m+Z\left( n\right) \right)
\right) =\overset{m}{\underset{j=0}{\sum }}\binom{m}{j}\left( -x\right)
^{m-j}\limfunc{E}f\left( T\circ Z\left( j+n\right) \right) ,  \label{m2}
\end{equation}%
where $\limfunc{E}Z$ means the expectation of the random variable $Z.$ 
\newline
In particular, for $y=0$ (respectively $n=y=0$), we get $T\left( n\right) =n$
and the identity (\ref{m2}) implies%
\begin{eqnarray}
\left( 1-x\right) ^{m}\limfunc{E}f\left( m+Z\left( n\right) \right) &=&%
\overset{m}{\underset{j=0}{\sum }}\binom{m}{j}\left( -x\right) ^{m-j}%
\limfunc{E}f\left( Z\left( j+n\right) \right) ,  \label{o2} \\
\left( 1-x\right) ^{m}f\left( m\right) &=&\overset{m}{\underset{j=0}{\sum }}%
\binom{m}{j}\left( -x\right) ^{m-j}\limfunc{E}f\left( Z\left( j\right)
\right) .  \label{p2}
\end{eqnarray}
\end{corollary}

\begin{proof}
For $k\in \left \{ 0,1,\ldots ,n\right \} $ we have%
\begin{eqnarray*}
P\left( W\left( n\right) =k\right) &=&P\left( T\circ Z\left( n\right)
=k\right) \\
&=&\overset{n}{\underset{j=0}{\sum }}P\left( T\left( j\right) =k/Z\left(
n\right) =j\right) P\left( Z\left( n\right) =j\right) \\
&=&\overset{n}{\underset{j=0}{\sum }}P\left( T\left( j\right) =k\right)
P\left( Z\left( n\right) =j\right) \\
&=&\overset{n}{\underset{j=0}{\sum }}\binom{j}{k}\left( 1-y\right)
^{k}y^{j-k}\binom{n}{j}\left( 1-x\right) ^{j}x^{n-j} \\
&=&\binom{n}{k}\overset{n}{\underset{j=k}{\sum }}\binom{n-k}{j-k}\left(
1-y\right) ^{k}y^{j-k}\left( 1-x\right) ^{j}x^{n-j} \\
&=&\binom{n}{k}\overset{n-k}{\underset{j=0}{\sum }}\binom{n-k}{j}\left(
1-y\right) ^{k}y^{j}\left( 1-x\right) ^{j+k}x^{n-k-j} \\
&=&\binom{n}{k}\left[ \left( 1-x\right) \left( 1-y\right) \right] ^{k}\left[
1-\left( 1-x\right) \left( 1-y\right) \right] ^{n-k}.
\end{eqnarray*}%
Since $S_{n}\left( x\right) =\limfunc{E}f\left( Z\left( n\right) \right) $
and $S_{n}\left( y\right) =\limfunc{E}f\left( T\left( n\right) \right) ,$
the identity (\ref{q1}) means that for any function $f$ defined on
nonnegative integers with $a_{k}=f\left( k\right) $ we have%
\begin{equation*}
\left( 1-x\right) ^{m}\overset{n}{\underset{k=0}{\sum }}\limfunc{E}f\left(
T\left( k+m\right) \right) \binom{n}{k}\left( 1-x\right) ^{k}x^{n-k}=\overset%
{m}{\underset{j=0}{\sum }}\binom{m}{j}\left( -x\right) ^{m-j}\limfunc{E}%
f\left( T\circ Z\left( j+n\right) \right) ,
\end{equation*}%
which is equivalent to the identity (\ref{m2}).
\end{proof}

\begin{remark}
Let $\left( x_{n};\ n\geq 0\right) $ be sequence of real numbers with $%
x_{n}\in \left] 0,1\right[ ,$ $n\geq 0,$ and let%
\begin{equation*}
S_{n}(x_{0})=\overset{n}{\underset{k=0}{\sum }}a_{k}\binom{n}{k}\left(
1-x_{0}\right) ^{k}x_{0}^{n-k}.
\end{equation*}%
Then by setting $q:=x_{1}\cdots x_{r},$ the identity (\ref{d1bis}) gives%
\begin{equation*}
S_{n}\left( 1-x_{0}x_{1}\cdots x_{r}\right) =S_{n}\left( 1-x_{0}q\right) =%
\overset{n}{\underset{k=0}{\sum }}S_{k}\left( q\right) \binom{n}{k}\left(
1-x_{0}\right) ^{n-k}x_{0}^{k},
\end{equation*}%
and the identity (\ref{d1}) can be written as:%
\begin{equation}
S_{n}\left( 1-x_{0}x_{1}\cdots x_{r}\right) =\overset{n}{\underset{k=0}{\sum 
}}S_{k}\left( 1-x_{1}\cdots x_{r}\right) \binom{n}{k}\left( 1-x_{0}\right)
^{n-k}x_{0}^{k}.  \label{u2}
\end{equation}%
Similarly, the identity (\ref{q1}) is equivalent to%
\begin{equation*}
x_{0}^{m}\overset{n}{\underset{k=0}{\sum }}S_{k+m}\left( 1-x_{1}\right) 
\binom{n}{k}\left( 1-x_{0}\right) ^{n-k}x_{0}^{k}=\overset{m}{\underset{j=0}{%
\sum }}\left( -1\right) ^{m-j}\binom{m}{j}\left( 1-x_{0}\right)
^{m-j}S_{j+n}\left( 1-x_{0}x_{1}\right) ,
\end{equation*}%
and by replacing $x_{1}$ by $x_{1}\cdots x_{r}$, it can be written as:%
\begin{equation}
x_{0}^{m}\overset{n}{\underset{k=0}{\sum }}S_{k+m}\left( 1-x_{1}\cdots
x_{r}\right) \binom{n}{k}\left( 1-x_{0}\right) ^{n-k}x_{0}^{k}=\overset{m}{%
\underset{j=0}{\sum }}\left( -1\right) ^{m-j}\binom{m}{j}\left(
1-x_{0}\right) ^{m-j}S_{j+n}\left( 1-x_{0}x_{1}\cdots x_{r}\right) .
\label{t2}
\end{equation}
\end{remark}

\begin{remark}
By replacing $q$ by $1/q$ in the identities (\ref{d1}) and(\ref{d1bis}) we
obtain the following alternating transforms%
\begin{equation}
\overset{n}{\underset{k=0}{\sum }}\left( -1\right) ^{k}a_{k}\binom{n}{k}%
\left( 1-q\right) ^{k}=q^{n}S_{n}\left( \frac{1}{q}\right) \text{ \ and \ }%
\overset{n}{\underset{k=0}{\sum }}\left( -1\right) ^{k}S_{k}\left(
1-x\right) \binom{n}{k}\left( 1-q\right) ^{n-k}=q^{n}S_{n}\left( 1-\frac{x}{q%
}\right) .  \label{a3}
\end{equation}
\end{remark}

\section{Applications to Appell polynomials}

Let $\left( c_{n};n\geq 0\right) $ be a sequence of real numbers such that $%
c_{0}=1$ and let $\left( f_{n}\left( y\right) ;\ n\geq 0\right) \ $ be a
sequence of Appell polynomials defined by \cite{App}:%
\begin{equation}
\sum_{n\geq 0}f_{n}\left( y\right) \frac{t^{n}}{n!}=F\left( t\right) e^{yt},
\label{e2}
\end{equation}%
where $F\left( t\right) =1+\sum_{n\geq 1}c_{n}\frac{t^{n}}{n!}.$ This
sequence can be also defined by%
\begin{equation*}
f_{0}\left( y\right) =1\text{ \ and \ }\frac{d}{dy}f_{n}\left( y\right)
=nf_{n-1}\left( y\right) .
\end{equation*}%
This shows that for $a_{n}=\frac{f_{n}\left( y\right) }{n!}$ in Proposition %
\ref{T1}, we get $\nabla a_{n}=\left( 1-D\right) \frac{f_{n}\left( y\right) 
}{n!},$ where $D=\frac{d}{dy}.$ \newline
Consequently, $M\left( n,j\right) =\nabla ^{j}a_{n}=\frac{1}{n!}\left(
1-D\right) ^{j}f_{n}\left( y\right) $ and 
\begin{equation}
\overset{n}{\underset{k=0}{\sum }}\binom{n}{k}\frac{n!}{k!}f_{k}\left(
y\right) \left( 1-q\right) ^{k}q^{n-k}=\left( 1-q+qD\right) ^{n}f_{n}\left(
y\right) .  \label{f2}
\end{equation}%
As special cases of Appell polynomials we cite Bernoulli and Euler
polynomials of high order defined by%
\begin{equation}
\sum_{n\geq 0}B_{n}^{\left( \alpha \right) }\left( y\right) \frac{t^{n}}{n!}%
=\left( \frac{t}{e^{t}-1}\right) ^{\alpha }e^{yt}\text{ \ and \ }\sum_{n\geq
0}E_{n}^{\left( \alpha \right) }\left( y\right) \frac{t^{n}}{n!}=\left( 
\frac{2}{e^{t}+1}\right) ^{\alpha }e^{yt}.  \label{g2}
\end{equation}

\begin{corollary}
Let $\left( f_{n}\left( x\right) \right) $ be a sequence of Appell
polynomials defined as above and let $\alpha ,\lambda \left( \neq 0\right) $
be real numbers. There holds%
\begin{equation}
\overset{n}{\underset{k=0}{\sum }}f_{n-k}\left( x\right) \binom{n}{k}\left(
1-q\right) ^{k}\left( \lambda q\right) ^{n-k}=\left( \lambda q\right)
^{n}f_{n}\left( x-\frac{1}{\lambda }+\frac{1}{\lambda q}\right) .  \label{k1}
\end{equation}%
In particular, we have%
\begin{eqnarray}
\overset{n}{\underset{k=0}{\sum }}B_{n-k}^{\left( \alpha \right) }\left(
x\right) \binom{n}{k}\left( 1-q\right) ^{k}\left( \lambda q\right) ^{n-k}
&=&\left( \lambda q\right) ^{n}B_{n}^{\left( \alpha \right) }\left( x-\frac{1%
}{\lambda }+\frac{1}{\lambda q}\right) ,  \label{m1} \\
\overset{n}{\underset{k=0}{\sum }}E_{n-k}^{\left( \alpha \right) }\left(
x\right) \binom{n}{k}\left( 1-q\right) ^{k}\left( \lambda q\right) ^{n-k}
&=&\left( \lambda q\right) ^{n}E_{n}^{\left( \alpha \right) }\left( x-\frac{1%
}{\lambda }+\frac{1}{\lambda q}\right) .  \label{n1}
\end{eqnarray}
\end{corollary}

\begin{proof}
For $a_{n}=\lambda ^{n}f_{n}\left( x\right) $ in Proposition \ref{T1} we
obtain%
\begin{equation*}
\overline{M}\left( j\right) =\overset{j}{\underset{l=0}{\sum }}\left(
-1\right) ^{l}\binom{j}{l}a_{l}=\overset{j}{\underset{l=0}{\sum }}\binom{j}{l%
}\left( -\lambda \right) ^{l}f_{l}\left( x\right) ,
\end{equation*}%
and since%
\begin{equation*}
\underset{j\geq 0}{\sum }\overline{M}\left( j\right) \frac{z^{j}}{j!}=e^{t}%
\underset{l\geq 0}{\sum }f_{l}\left( x\right) \frac{\left( -\lambda t\right)
^{l}}{l!}=F\left( -\lambda t\right) e^{\left( x-\frac{1}{\lambda }\right)
\left( -\lambda t\right) }=\underset{j\geq 0}{\sum }f_{j}\left( x-\frac{1}{%
\lambda }\right) \frac{\left( -\lambda t\right) ^{j}}{j!},
\end{equation*}%
it follows that $\overline{M}\left( j\right) =\left( -\lambda \right)
^{j}f_{j}\left( x-\frac{1}{\lambda }\right) .$ Then, from Proposition \ref%
{T1} we get%
\begin{equation*}
\overline{S}_{n}\left( q\right) =\overset{n}{\underset{j=0}{\sum }}\binom{n}{%
j}f_{j}\left( x-\frac{1}{\lambda }\right) \left( \lambda q\right) ^{j},
\end{equation*}%
and because%
\begin{eqnarray*}
\underset{n\geq 0}{\sum }\overline{S}_{n}\left( q\right) \frac{t^{n}}{n!} &=&%
\underset{j\geq 0}{\sum }f_{j}\left( x-\frac{1}{\lambda }\right) \frac{%
\left( \lambda qt\right) ^{j}}{j!}\underset{n\geq j}{\sum }\frac{t^{n-j}}{%
\left( n-j\right) !} \\
&=&e^{t}\underset{j\geq 0}{\sum }f_{j}\left( x-\frac{1}{\lambda }\right) 
\frac{\left( \lambda qt\right) ^{j}}{j!} \\
&=&F\left( \lambda qt\right) e^{\left( x-\frac{1}{\lambda }\right) \left(
\lambda qt\right) +t} \\
&=&F\left( \lambda qt\right) e^{\left( x-\frac{1}{\lambda }+\frac{1}{\lambda
q}\right) \left( \lambda qt\right) } \\
&=&\underset{n\geq 0}{\sum }f_{n}\left( x-\frac{1}{\lambda }+\frac{1}{%
\lambda q}\right) \frac{\left( \lambda qt\right) ^{n}}{n!},
\end{eqnarray*}%
it follows $\overline{S}_{n}\left( q\right) =\left( \lambda q\right)
^{n}f_{n}\left( x-\frac{1}{\lambda }+\frac{1}{\lambda q}\right) ,$ which
completes the proof.
\end{proof}

\noindent By setting $q=\frac{x}{x+y}$ in (\ref{c}) we get%
\begin{eqnarray}
\overset{n}{\underset{k=0}{\sum }}a_{k}\binom{n}{k}y^{k}x^{n-k} &=&\overset{n%
}{\underset{j=0}{\sum }}\left( -1\right) ^{j}\binom{n}{j}M\left( n,j\right)
\left( x+y\right) ^{n-j}x^{j},  \label{f} \\
\overset{n}{\underset{k=0}{\sum }}a_{k}\binom{n}{k}y^{n-k}x^{k} &=&\overset{n%
}{\underset{j=0}{\sum }}\left( -1\right) ^{j}\binom{n}{j}\overline{M}\left(
j\right) \left( x+y\right) ^{j}x^{n-j}.  \label{fbis}
\end{eqnarray}

\noindent Now, if we denote by $\mathbf{C}^{n}$ for the umbra defined by $%
\mathbf{C}^{n}:=c_{n},$ then we can write%
\begin{equation}
F\left( t\right) =1+\sum_{n\geq 1}c_{n}\frac{t^{n}}{n!}=1+\sum_{n\geq 1}%
\mathbf{C}^{n}\frac{t^{n}}{n!}=\exp \left( \mathbf{C}t\right) ,  \label{h2}
\end{equation}%
which gives the following representation%
\begin{equation*}
\sum_{n\geq 0}f_{n}\left( y\right) \frac{t^{n}}{n!}=\exp \left( \left( 
\mathbf{C}+y\right) t\right) =1+\sum_{n\geq 1}\left( \mathbf{C+}y\right) ^{n}%
\frac{t^{n}}{n!}.
\end{equation*}%
So, $f_{n}\left( y\right) $ can be written by the umbral representation as%
\begin{equation}
f_{n}\left( y\right) =\left( \mathbf{C}+y\right) ^{n}.  \label{t}
\end{equation}%
By using the umbral technic used in \cite{Mih2}, the identity (\ref{f}) can
be generalized as follows.

\begin{corollary}
For any sequence $\left( f_{n}\left( y\right) \right) $ of Appell
polynomials there hold%
\begin{eqnarray}
\overset{n}{\underset{k=0}{\sum }}a_{k}\binom{n}{k}f_{k}\left( y\right)
x^{n-k} &=&\overset{n}{\underset{j=0}{\sum }}\left( -1\right) ^{j}\binom{n}{j%
}M\left( n,j\right) f_{n-j}\left( x+y\right) x^{j},  \label{u} \\
\overset{n}{\underset{k=0}{\sum }}a_{k}\binom{n}{k}f_{n-k}\left( y\right)
x^{k} &=&\overset{n}{\underset{j=0}{\sum }}\left( -1\right) ^{j}\binom{n}{j}%
\overline{M}\left( j\right) f_{j}\left( x+y\right) x^{n-j}.  \label{ubis}
\end{eqnarray}
\end{corollary}

\begin{proof}
In the identity (\ref{f}), we replace $y$ by $\mathbf{C}+y$ to obtain%
\begin{eqnarray*}
\overset{n}{\underset{k=0}{\sum }}a_{k}\binom{n}{k}\left( \mathbf{C}%
+y\right) ^{k}x^{n-k} &=&\overset{n}{\underset{j=0}{\sum }}\left( -1\right)
^{j}\binom{n}{j}M\left( n,j\right) \left( \mathbf{C+}x+y\right) ^{n-j}x^{j},
\\
\overset{n}{\underset{k=0}{\sum }}a_{k}\binom{n}{k}\left( \mathbf{C}%
+y\right) ^{n-k}x^{k} &=&\overset{n}{\underset{j=0}{\sum }}\left( -1\right)
^{j}\binom{n}{j}\overline{M}\left( j\right) \left( \mathbf{C+}x+y\right)
^{j}x^{n-j}.
\end{eqnarray*}%
Hence, the desired identities follow by using the umbral identity $\left( 
\mathbf{C}+y\right) ^{n}=f_{n}\left( y\right) .$
\end{proof}

\noindent As special case of the last Corollary, for $a_{n}=F_{n+r}$ we get%
\begin{equation}
\overset{n}{\underset{k=0}{\sum }}\binom{n}{k}F_{k+r}f_{k}\left( y\right)
x^{n-k}=\overset{n}{\underset{j=0}{\sum }}\left( -1\right) ^{j}\binom{n}{j}%
F_{n+r-2j}f_{n-j}\left( x+y\right) x^{j}  \label{v}
\end{equation}%
which gives for $r=n:$%
\begin{equation}
\overset{n}{\underset{k=0}{\sum }}\binom{n}{k}F_{n+k}f_{k}\left( y\right)
x^{n-k}=\overset{n}{\underset{j=0}{\sum }}\binom{n}{j}F_{2j}f_{j}\left(
x+y\right) x^{n-j}.  \label{o1}
\end{equation}%
Also, for $a_{n}=\left( -1\right) ^{n}\binom{\alpha }{n+r}$ we obtain the
following identities on Appell polynomials%
\begin{eqnarray}
\overset{n}{\underset{k=0}{\sum }}\left( -1\right) ^{k}\binom{\alpha }{r+k}%
\binom{n}{k}f_{k}\left( y\right) x^{n-k} &=&\overset{n}{\underset{j=0}{\sum }%
}\left( -1\right) ^{n-j}\binom{n}{j}\binom{j+\alpha }{n+r}f_{n-j}\left(
x+y\right) x^{j},  \label{w} \\
\overset{n}{\underset{k=0}{\sum }}\left( -1\right) ^{k}\binom{\alpha }{r+k}%
\binom{n}{k}f_{k}\left( y\right) x^{n-k} &=&\overset{n}{\underset{j=0}{\sum }%
}\left( -1\right) ^{j}\binom{n}{j}\binom{j+\alpha }{j+r}f_{j}\left(
x+y\right) x^{n-j}.  \label{wbis}
\end{eqnarray}

\begin{corollary}
For any sequence $\left( f_{n}\left( y\right) \right) $ of Appell
polynomials there holds%
\begin{equation}
\overset{n}{\underset{k=0}{\sum }}a_{k}\binom{n}{k}\left( bx\right)
^{k}f_{n-k}\left( y+\left( 1-b\right) x\right) =\overset{n}{\underset{k=0}{%
\sum }}S_{k}\left( 1-b\right) \binom{n}{k}x^{k}f_{n-k}\left( y\right) .
\label{s2}
\end{equation}
\end{corollary}

\begin{proof}
Similarly as above, the identity (\ref{d1bis}) can be written as%
\begin{equation*}
S_{n}\left( 1-bq\right) =\overset{n}{\underset{k=0}{\sum }}a_{k}\binom{n}{k}%
\left( 1-bq\right) ^{n-k}\left( bq\right) ^{k}=\overset{n}{\underset{k=0}{%
\sum }}S_{k}\left( 1-b\right) \binom{n}{k}\left( 1-q\right) ^{n-k}q^{k}
\end{equation*}%
and by setting $q=\frac{x}{x+y},$ it becomes%
\begin{equation*}
\overset{n}{\underset{k=0}{\sum }}a_{k}\binom{n}{k}\left( y+\left(
1-b\right) x\right) ^{n-k}\left( bx\right) ^{k}=\overset{n}{\underset{k=0}{%
\sum }}S_{k}\left( 1-b\right) \binom{n}{k}y^{n-k}x^{k}.
\end{equation*}%
So, by replacing $y$ by $\mathbf{C}+y$ and using (\ref{t}), we get the
desired identity.
\end{proof}


\begin{thebibliography}{99}
\bibitem{App} P. Appell, Sur une classe de polyn\^{o}mes. \textit{Ann. Sci.
Ecole Norm. Sup.,} 9, (1880), 119--144.

\bibitem{Ber} S. Bernstein, D\'{e}monstration du th\'{e}oreme de
Weierstrass, fond\'{e}e sur le calcul des probabilit\'{e}s. \textit{Comm.
Kharkov Math. Soc.,} (1912), 1--2.

\bibitem{Boy} K. N. Boyadzhiev, Harmonic number identities via Euler's
transform. \textit{J. Integer Seq.,} (2009), 12, 8.

\bibitem{Boy1} K. N. Boyadzhiev, Exponential polynomials, Stirling numbers,
and evaluation of some gamma integrals. \textit{Abstr. Appl. Anal.,} (2009),
1--18. https://doi.org/10.1155/2009/168672

\bibitem{Com} L. Comtet, \textit{Advanced Combinatorics}, D. Reidel,
Dordrecht, Holland, 1974.

\bibitem{Kom} T. Komatsu, B. Sury, Polynomial identities for binomial sums
of harmonic numbers of higher order. \textit{Mathematics,} 13, 321, (2025).
https://doi.org/10.3390/math13020321

\bibitem{Cic} J. Cicho\'{n}, Z. Golebiewski, On Bernoulli sums and Bernstein
polynomials, in 23rd Intern. \textit{Meeting on Probabilistic,
Combinatorial, and Asymptotic Methods for the Analysis of Algorithms},
(2012), 179--190.

\bibitem{Rob} R. S. Costas-Santos, The Laguerre constellation of classical
orthogonal polynomials. \textit{Mathematics,} 13, 277, (2025).
https://doi.org/10.3390/math13020277

\bibitem{Dil} A. Dil, V. Kurt, Investigating geometric and exponential
polynomials with Euler-Seidel matrices. \textit{J. Integer Seq.}, 14,
(2011), Article 11.4.6.

\bibitem{Ges} I. Gessel, Lagrange Inversion. \textit{J. Combin. Theory Ser.
A,} 144, (2016), 212--249.

\bibitem{Fla} P. Flajolet, Singularity analysis and asymptotics of Bernoulli
sums. \textit{Theor. Comput. Sci.,} 215(1-2), (1999), 371--381.

\bibitem{Jac} P. Jacquet, W. Szpankowski, Entropy computations via analytic
depoissonization. \textit{IEEE Trans. Inf. Theory.,} 45(4), (2002),
1072--1081.

\bibitem{Goy} T. Goy, M. Shattuck, Determinants of Toeplitz-Hessenberg
matrices with Fuss-Catalan entries. \textit{J. Integer Seq.}, 29, (2026),
Article 26.1.2.

\bibitem{Gou} H. W. Gould, Series transformations for finding recurrences
for sequences. \textit{Fib. Quart.} 28(2), (1990),166-71.

\bibitem{Hau} P. Haukkanen, Formal power series for binomial sums of
sequences of numbers. \textit{Fib. Quart.,} 31(1), (1993), 28-31.

\bibitem{Mih2} M. Mihoubi, S. Taharbouchet, Some identities involving Appell
polynomials. \textit{Quaest. Math.,} 43(2), (2020), 203--212.

\bibitem{Mih3} M. Mihoubi, S. Taharbouchet, Identities and congruences
involving the geometric polynomials. \textit{Miskolc math. notes,} 20(1),
(2019), 395--408.

\bibitem{Mne} S. A. Mneimneh, Binomial sum of harmonic numbers. \textit{%
Discrete Math.,} 346(6), (2023). https://doi.org/10.1016/j.disc.2022.113075

\bibitem{Spi} M. Spivey, Combinatorial sums and finite differences. \textit{%
Discrete Math.,} 307 (2007), 3130-3146.

\bibitem{Sur} E. Surya, L. Warnke, Lagrange inversion formula by induction. 
\textit{Amer. Math. Monthly}, 130, (2023), 944-948.
\end{thebibliography}
\end{document}